\documentclass[11pt]{article}

\usepackage[a4paper,margin=1.1in]{geometry}
\usepackage{setspace}
\usepackage[english]{babel}
\usepackage[T1]{fontenc}
\usepackage{lmodern}
\usepackage{microtype}

\usepackage{amsmath,amssymb,amsthm}
\usepackage{mathtools}
\usepackage{mathtools}
\usepackage{mathrsfs}
 \usepackage{pgf,tikz}
 \usetikzlibrary{automata}
 \usetikzlibrary{arrows}
\usepackage{blkarray}
\usepackage[numbers]{natbib}

\usepackage[hidelinks]{hyperref}

\usepackage{fancyhdr}
\theoremstyle{plain}
\newtheorem{theorem}{Theorem}[section]
\newtheorem{lemma}[theorem]{Lemma}
\newtheorem{proposition}[theorem]{Proposition}
\newtheorem{corollary}[theorem]{Corollary}
\newtheorem{problem}[theorem]{Problem}
\theoremstyle{definition}
\newtheorem{definition}[theorem]{Definition}
\newtheorem{example}[theorem]{Example}

\newtheorem*{mtheorem}{Main Theorem}

\theoremstyle{remark}
\newtheorem{remark}[theorem]{Remark}

\title{\textbf{On the non-existence of certain maximizing curves of degree $7$}}
\author{
   Łukasz Merta, Marcin Zieli\'nski
}
\date{\today}

\begin{document}
\maketitle

\begin{abstract}
In this paper we show there does not exist any arrangement of degree $7$ consisting of $k\geq 1$ conics and $d\geq 1$ lines admitting certain naturally chosen ${\rm ADE}$-singularities. 
\end{abstract}


\medskip


\section{Introduction}\label{s:introduction}
In this paper, we study the geometry of maximizing plane curves of degree $7$. From a historical perspective, the geometry of maximizing curves of even degree is a classical topic in the theory of algebraic surfaces, dating back to the work of Persson \cite{Persson}, who investigated the existence of such curves in connection with the construction of algebraic surfaces with large Picard numbers. More recently, this class of curves has been studied by Dimca and Pokora \cite{DimPok} in the context of so-called free curves. They showed that maximizing curves of even degree are free and attain the maximal possible global Tjurina number for a given degree. Another contribution of Dimca and Pokora is the introduction of a new class of curves, namely maximizing curves of odd degree, which we now define. Let $S = \mathbb{C}[x,y,z]$ be the coordinate ring of $\mathbb{P}^{2}_{\mathbb{C}}$ and for a polynomial $f \in S_{d}$ we define by $J_{f}$ the Jacobian ideal, i.e., the ideal generated by the partials $f_x, f_y, f_z$.
\begin{definition}
Let $C\subset \mathbb{P}^{2}_{\mathbb{C}}$ be a reduced curve of odd degree $d=2m+1$ with $m \geq 2$. Assume that $C$ admits only ${\rm ADE}$-singularities. We say that $C$ is maximizing curve if
$${\rm deg}(J_{f}) = 3m^{2}+1.$$
\end{definition}
\noindent
It is worth noting here that for reduced plane curves one has
$${\rm deg}(J_{f})= \tau(C),$$ where $\tau(C)$ is the global Tjurina number of $C$ which is defined as $\tau(C) = \sum_{p \in {\rm Sing}(C)}\tau_{p}$ with ${\rm Sing}(C)$ being the set of all singular points of $C$. It turns out, perhaps surprisingly, that maximizing plane curves of odd degree are extremely rare. In degree $d=5$, maximizing curves are completely classified: there are exactly four such curves, a result that follows as a straightforward consequence of the work of C.\,T.\,C.~Wall \cite[Section 5]{Wall}. In degree $d=7$, only two examples are currently known; both arise from very special curve arrangements \cite[Example 5.10]{DimPok}. For degree $d=9$, we are aware of just a single example of a maximizing curve \cite[Example 6.4]{DIPS}.  

Taken together, these observations suggest a striking scarcity of maximizing curves in odd degrees 
and naturally lead to the following research problem.
\begin{problem}
Construct examples of maximizing curves of odd degree $d \geq 11$.
\end{problem}
Here, we do not address the above problem. Instead, we return to the case of degree $d=7$ in order to understand the geometric constraints governing the simplest possible maximizing curve of degree $7$. The simplest possible scenario is to study maximizing line arrangements $\mathcal{L}$, i.e., these are arrangements admitting only double and triple intersection points. However, if $\mathcal{L}$ is maximizing, then ${\rm deg}(\mathcal{L}) \leq 6$. This means that, to construct maximizing curves of degree seven, we must consider arrangements consisting of more than just lines. This suggests considering arrangements that also contain conics. It turns out that there is one particularly nice maximizing curve of degree seven, namely Persson's arrangement consisting of $k=3$ conics and one line and it has two $D_{10}$ singularities, two $A_{1}$ singularities, and one $D_{6}$ singularity. On the other hand, by \cite[Theorem 3.1]{JanLes}, if $C\subset \mathbb{P}^{2}_{\mathbb{C}}$ is a reduced plane curve of degree $d=7$ admitting only $A_{t}$ singularities with $t \leq 12$, $D_l$ singularities for $l \leq 7$, $E_6$ singularities, then $C$ is never free. This suggests to formulate the following problem.
\begin{problem}
Is it possible to construct a maximizing conic-line arrangement of degree seven admitting, at worst, $A_{t}$ singularities with $t \in \{1, 3, 5, 7\}$ and $D_{l}$ singularities with $l \in \{4,6,8\}$?
\end{problem}

\noindent The main result of our paper is a \textbf{negative answer} to the above problem, see \hyperref[maintheorem]{Main Theorem}.

\vspace{8pt}
Here is the outline of our paper. We start with algebraic and geometric preliminaries devoted to tools and techniques that are used in the world of free and maximizing curves. Then we present our proof. Since it is a technical work, the proof will be divided into several steps, but we hope that our approach will allow us to keep maximal clarity of the presentation.
\section{Algebraic and geometric preliminaries}
We start with algebraic preliminaries. As indicated in the Introduction, we denote by $S$ the ring of polynomials in the three variables $x,y,z$ with complex coefficients. We will need the following important invariant that is attached to the syzygies of $J_{f}$.
\begin{definition}
Consider the graded $S$-module of Jacobian syzygies of $f$, namely $$AR(f)=\{(a,b,c)\in S^3 : af_x+bf_y+cf_z=0\}.$$
The minimal degree of non-trivial Jacobian relations for $f$ is defined as
$${\rm mdr}(f):=\min\{r : AR(f)_r\neq 0\}.$$ 
\end{definition}

Now we can introduce the notion of a free curve.
\begin{definition}
We say that a reduced curve $C = \{f = 0\}$ in the complex projective plane is \textbf{free} if the Jacobian ideal $J_{f}$ associated with $f$ is saturated with respect to the irrelevant ideal $\mathfrak{m} = \langle x, y, z \rangle$. 
\end{definition}
\begin{remark}
Using the language of modules, the freeness of $C = \{ f = 0 \}$ is equivalent to the fact that
$\mathrm{AR}(f)$ is a free $S$-module of rank $2$.
\end{remark}
We have the following nice criterion that focuses on the total Tjurina number and the minimal degree of the Jacobian relations \cite{duP}.
\begin{theorem}[Freeness criterion]
Let $C = \{ f=0 \}$ be a reduced curve in $\mathbb{P}^{2}_{\mathbb{C}}$ of degree $d$. Then the curve $C$  is free if and only if $r = {\rm mdr}(f)\leq (d-1)/2$ and
\begin{equation}
\label{duPles}
(d-1)^{2} - r(d-r-1) = \tau(C).
\end{equation}
\end{theorem}

Next, let us recall the classification of ${\rm ADE}$ singularities for curves by presenting their local normal forms captured from \cite{arnold}, namely
\begin{center}
\begin{tabular}{ll}
$A_{k}$ with $k\geq 1$ & $: \, x^{2}+y^{k+1}  = 0$, \\
$D_{k}$ with $k\geq 4$ & $: \, y^{2}x + x^{k-1}  = 0$,\\
$E_{6}$ & $: \, x^{3} + y^{4} = 0$, \\
$E_{7}$ & $: \, x^{3} + xy^{3} = 0$, \\
$E_{8}$ & $:\, x^{3}+y^{5} = 0$.
\end{tabular}
\end{center}
All reduced plane curves that admit only ${\rm ADE}$ singularities are called \textbf{simply singular curves} in the literature. Now we define our main object of study.
\begin{definition}
Let $C = \{ f=0 \}$ be a reduced simply singular curve in $\mathbb{P}^{2}_{\mathbb{C}}$ of \textbf{odd degree} $n=2m+1\geq 5$. We say that $C$ is a \textit{maximizing} curve if 
$$\tau(C) = 3m^{2}+1.$$
\end{definition}
\begin{remark}
It is worth recalling that the property of being a maximizing curve $C = \{f=0\}$ of degree $2m+1$ implies that this curve is free. Moreover, the maximality of $C$ implies that ${\rm mdr}(f)=m-1$ (see \cite{DimPok}).
\end{remark}
\begin{remark}
Observe that arrangements of lines and conics do not produce irreducible singularities
such as \(A_{2}\), \(A_{4}\), \(A_{6}\), \(E_{6}\), or \(E_{8}\), our choice of
singularities is minimal within this class of curves. Based on this observation, we will say that a reduced plane curve $C$ has only ${\rm AD}$-singularities meaning that $C$ does not admit $E_{j}$-singular points.
\end{remark}
Now we introduce the notion of a \emph{weak combinatorial type}, the concept of \emph{combinatorics} for plane curves, and the so-called \emph{B\'ezout constructability property} for a reduced plane curve $C$. Let us now explain the general principles behind these notions.

\begin{itemize}
    \item A \emph{weak combinatorial type} captures the basics of a curve without keeping track of all detailed geometric information.
    \item The \emph{combinatorics} of a curve describes how its irreducible components intersect and how singular points are distributed, together with decorations coming from the intersection indices. 
    \item \emph{B\'ezout constructability property} is a condition that ensures certain configurations of points and intersections on $C$ are possible, consistent with B\'ezout's Theorem.
\end{itemize}
\begin{definition}
Let $C = \{C_{1}, \ldots, C_{k}\}$ be a reduced plane curve whose irreducible components are all smooth.
We define the \emph{weak combinatorial type} of $C$ as
\[
C(d_{1}, \ldots, d_{s}; m_{1}, \ldots, m_{p}),
\]
where $d_{j}$ denotes the number of irreducible components of $C$ of degree $j$, and
$m_{i}$ denotes the number of singular points of $C$ having a fixed topological type $M_{i}$.
\end{definition}
\begin{definition}
\label{cs}
Let $C = \{C_{1}, \ldots, C_{k}\}$ be a reduced plane curve whose irreducible components are all smooth.
\begin{enumerate}
\item[(a)]
The \textbf{combinatorics} $CS$ for $C$ is defined as a
\textbf{decorated Levi graph}, that is, a bipartite graph with the following structure:
\begin{itemize}
\item Vertices $x_{1}, \ldots, x_{k}$ correspond to the irreducible components $C_{1}, \ldots, C_{k}$.
Each vertex $x_{i}$ is assigned the weight $\deg(C_{i})$.
\item Vertices $y_{1}, \ldots, y_{s}$ correspond to the intersection points of $C$.
\item A vertex $y_{j}$ is incident to a vertex $x_{i}$ if and only if the point $p_{j} \in C_{i}$.
\item Each vertex $y_{j}$ is decorated by the collection of intersection numbers
\[
\bigg\{\, i_{p_{j}}(C_{i_{1}}, C_{i_{2}}) \,\bigg\},
\]
where the indices $i_{1} < i_{2}$ range over all pairs such that
$p_{j} \in C_{i_{1}} \cap C_{i_{2}}$.
\end{itemize}

\item[(b)]
We say that the curve $C$, endowed with the combinatorics $CS$, is
\textbf{B\'ezout constructible} if for every pair of indices
$1 \leq i_{1} < i_{2} \leq k$ one has
\[
\deg(C_{i_{1}})\cdot \deg(C_{i_{2}})
=
\sum_{p \in \mathrm{Sing}(C)}
i_{p}(C_{i_{1}}, C_{i_{2}}).
\]
\end{enumerate}
\end{definition}
\begin{remark}
\begin{enumerate}
\item In the setting of Definition~\ref{cs}, the B\'ezout constructability property is combinatorial since it can be checked using only the data contained in the combinatorics~$CS$, i.e., the data encoded in the decorated Levi graph.
\item The fact that a given combinatorics $CS$ satisfies the B\'ezout constructability property does not necessarily imply that $CS$ is geometrically realizable via a plane curve in $\mathbb{P}^{2}_{\mathbb{C}}$.
\end{enumerate}
\end{remark}
\begin{example}
Let us consider the following combinatorial situation: we have three smooth conics $C_{1}, C_{2}, C_{3}$ such that they have three singular points of type $A_{3}$ and two points of type $A_{5}$. Observe that this combinatorics is not B\'ezout constructible since there always exists (after a suitable relabeling of indices indicating curves and singular points) a situation in which
$$4 = {\rm deg}(C_{1}) \cdot {\rm deg}(C_{2}) \neq \sum_{p \in {\rm Sing}(C)} i_p(C_1,C_2) = i_{p_{1}}(C_{1},C_{2}) + i_{p_{2}}(C_{1},C_{2}) = 2+3 = 5,$$
meaning that $p_{1} \in C_{1}\cap C_{2}$ is an $A_{3}$-singularity and $p_{2} \in C_{1} \cap C_{2}$ is an $A_{5}$-singularity.
\end{example}
The following definition will be crucial in our combinatorial considerations.
\begin{definition}\label{def: PIM}
Let $C=\{C_1,\ldots,C_k\}\subset \mathbb{P}^2_{\mathbb{C}}$
be a reduced plane curve whose irreducible components $C_i$ are smooth of degree $\deg(C_i)=d_i$.
Assume that $C$ has only $\mathrm{AD}$-singularities. Let
\[
\mathrm{Sing}(C)=\{p_1,\ldots,p_s\}
\]
be the set of singular points. To each singular point $p_j$ we associate its topological type $W_j$ with $W \in \{A,D\}$.

We define the \emph{pair--incidence matrix} of $C$, denoted by $\mathrm{PIM}(C)$, as the matrix with $\binom{k}{2}$ rows and $s$ columns defined as follows:
\begin{itemize}
    \item the rows are indexed by  pairs $C_aC_b$ with $a<b$.
    \item the columns are indexed by the singular points $p_j$ (labeled by their type $W_j$).
\end{itemize}
The entry corresponding to the pair $(C_a,C_b)$ and the point $p_j$ is given by
\[
\mathrm{PIM}(C)_{(a,b),j}
=
i_{p_j}(C_a,C_b),
\]
the local intersection multiplicity of $C_a$ and $C_b$ at $p_j$.
\end{definition}

\begin{remark}
    The arrangement $C = \{C_1, \dots, C_k\} \subset \mathbb{P}^2_{\mathbb{C}}$ with $\mathrm{Sing}(C)=\{p_1,\ldots,p_s\}$ has the B\'ezout constructability property if and only if in the $\mathrm{PIM}(C)$ matrix
    \begin{itemize}
        \item sum of values in the row labeled by $C_aC_b$ is equal to $\deg(C_a)\cdot\deg(C_b)$ for each $a < b$,
        \item for each column corresponding to a given singularity, the nonzero entries are determined by the type of the singularity. For example, a singularity of type $D_{8}$ has three smooth branches, which intersect pairwise with multiplicities $1$, $1$, and $3$. Consequently, these are precisely the values that must occur as the nonzero entries in the corresponding column.
    \end{itemize}
\end{remark}

\begin{example}
Let us consider the following arrangment of three smooth conics $C = \{C_{1}$, $C_{2}, C_{3}\}$ having exactly four singular points, namely $p_{1}, p_{2}, p_{3}$ are $A_{5}$-singularities and $p_{4}$ is a $D_{4}$-singularity. The corresponding ${\rm PIM}(\mathcal{C})$ matrix has the following shape.
\[
\begin{blockarray}{ccccc}
 & A_{5} & A_{5} & A_5 & D_4\\
\begin{block}{c|cccc|}
  C_{1}C_{2} & 3 & 0 & 0 & 1 \\
  C_{1}C_{3} & 0 & 3 & 0 & 1 \\
  C_{2}C_{3} & 0 & 0 & 3 & 1 \\
\end{block}
\end{blockarray}.
 \]
We can check that the B\'ezout constructability property holds for~$C$. Moreover, this arrangement is geometrically constructible over the complex numbers; see \cite[Remark~4.2.3]{Celal}.
\end{example}
\noindent
We now focus on the class of curves that will be studied in this work. 
Specifically, we consider plane curves of degree $7$ consisting of $k \geq 1$ conics 
and $d \geq 1$ lines. Since our analysis is restricted to a particular class of singularities, we impose the following assumption.
\begin{center}
\textbf{Assumption} ($\square$): our conic-line arrangements admit, at worst, 
$n_{2}$ singularities of type $A_{1}$, $n_{3}$ singularities of type $D_{4}$, 
$t_{p}$ singularities of type $A_{p}$ with $p \in \{3,5,7\}$,\\  and $d_{t}$ 
singularities of type $D_{t}$ with $t \in \{6,8\}$. 
\end{center}
Here we follow a well-established convention that the number of ordinary singularities of multiplicity $r$ is denoted by $n_{r}$. Moreover, observe that arrangements consisting only of conics and lines cannot produce singularities of type $D_{5}$ or $D_{7}$ -- for this reason, such singularities are excluded from our considerations. 
\begin{remark}
In the setting of our paper, which is considerably more restrictive than the general case, the weak combinatorics of a curve $C$ is encoded by a vector of non-negative integers of the form:
$$C(d,k;n_{2},n_{3},t_{3},t_{5},t_{7},d_{6},d_{8}).$$
\end{remark}
Next, we present main technical tools that will help us to exclude certain combinatorial types of curves. First, let us recall the following result \cite[Theorem A.1]{IC}.
\begin{theorem}[Pokora]
\label{HP}
Let $\mathcal{C} \subset \mathbb{P}^{2}_{\mathbb{C}}$ be an arrangement consisting of $k\geq 1$ smooth conics and $d\geq 1$ lines with ${\rm deg}(\mathcal{C}) = 2k+d \geq 7$. Assume that $\mathcal{C}$ admits singularities no worse than those in Assumption~$(\square)$. Then one has
\begin{equation}
\label{hib}
8k + n_{2} + \frac{3}{4}n_{3} \geq d + \frac{5}{2}t_{3} + 5t_{5} + \frac{29}{4}t_{7} +\frac{13}{8} d_{6} + \frac{15}{4} d_{8}.
\end{equation}
\end{theorem}
\noindent
Finally, let us present our main tool in our considerations that was proved in \cite{JanLes}.
\begin{theorem}
\label{main}
Let $C = \{ f=0 \}$ be a reduced simply singular curve in $\mathbb{P}^{2}_{\mathbb{C}}$ of \textbf{odd degree} $d=2m+1$ for $m \geq 3$ such that $C$ admits
\begin{itemize}
    \item singularities of type $A_k$ for $k \leq 4m$,
    \item singularities of type $D_l$ for $l \leq 2m+1$, and
    \item singularities of type $E_6$.
\end{itemize}
Then $C$ \textbf{cannot be} maximizing.
\end{theorem}
We will use the above results in our considerations. In particular, we will be able to exclude most of admissible weak-combinatorial types in our considerations.

\section{Proof of our main result}
Our first aim is to enumerate all weak combinatorial types of conic–line arrangements realizing the prescribed singularities. This problem reduces to solving the following system of Diophantine equations:
$$
\begin{cases}
n_{2} + 3t_{3} + 4n_{3} + 5t_{5} + 6d_{6} + 7t_{7} + 8d_{8} = \tau(\mathcal{C}) = 28,\\
n_{2} + 2t_{3} + 3n_{3} + 3t_{5} + 4t_{7} + 4d_{6} + 5d_{8} = \binom{7}{2}-k = 21 - k.
\end{cases}
$$
Using a computer-assisted search, we find that there are exactly \textbf{296} (na\"ive) weak-combinatorial types. These are listed in the Appendix at the end of the paper. We can significantly narrow the search space by applying Theorem~\ref{main}. For instance, if a given combinatorics does not contain singularities of type \(D_{8}\), then it cannot correspond to a maximizing curve of degree seven. Our classification is naturally stratified by the number of lines (or conics) that our arrangement can possess.

\paragraph{Scenario 1: five lines and one conic.} 
According to our list in the Appendix, there are exactly $103$ weak combinatorial types of arrangements consisting of one conic and five lines. Using Theorem~\ref{HP} and Theorem~\ref{main}, we can narrow this number down to just $5$ cases, namely:
\begin{multline}
\label{lc3}
C(n_{2},n_{3},t_{3},t_{5},t_{7},d_{6},d_{8}) \in \bigg\{(0,5,0,0,0,0,1), (1,4,1,0,0,0,1), (2,3,0,0,0,1,1), \\(4,1,0,0,0,2,1), (4,2,0,0,0,0,2) \bigg\}.
\end{multline}
The following structural result allows us to exclude all the listed cases in \eqref{lc3}.
\begin{theorem}
\label{s1}
Let $\mathcal{CL}$ be an arrangement of one conic and five lines with singularities no worse than those in Assumption~$(\square)$. Then $\mathcal{CL}$ cannot have $D_{8}$ singularities.
\end{theorem}
\begin{proof}
Although this argument may be well-known to experts, we briefly include a proof for completeness.
Working with our restricted class of ${\rm ADE}$-singularities, each singular point
$p \in \mathcal{CL}$ can be locally decomposed into pairs of smooth curves intersecting at $p$
with prescribed intersection indices $i_p$.
To obtain a singularity of type $D_{8}$, one needs three smooth curves,
say $\{C_1,C_2,C_3\}$, such that, up to relabeling,
\[
i_p(C_1,C_2)=1,\qquad i_p(C_2,C_3)=1,\qquad i_p(C_1,C_3)=3.
\]
Since $\mathcal{CL}$ consists only of lines and one conic, B\'ezout’s Theorem rules out the existence
of any intersection point with the intersection index equal to $3$, which completes the proof.
\end{proof}
This observation allows us to conclude this paragraph by the following.

\begin{proposition}\label{result1}
There is no maximizing arrangement of one conic and five lines admitting the singularities listed in Assumption~$(\square)$.
\end{proposition}

\paragraph{Scenario 2: three lines and two conics.}
According to our list in the Appendix, there are exactly \(116\) weak-combinatorial types of two conics and three lines. Using Theorem~\ref{HP} and Theorem~\ref{main}, we can narrow down the number of cases to consider which are enumerated below.  We are left with $33$ cases.
\begin{multline*}
C(n_{2},n_{3},t_{3},t_{5},t_{7},d_{6},d_{8}) \in \bigg\{
(0,2,0,0,0,2,1),(0,2,2,0,0,1,1),(0,2,4,0,0,0,1),\\(0,3,0,0,0,0,2),
(0,3,1,1,0,0,1),(1,1,1,0,0,2,1),(1,1,3,0,0,1,1),(1,2,0,1,0,1,1),\\(1,2,1,0,0,0,2),
(1,2,2,1,0,0,1),(1,3,0,0,1,0,1),(2,0,0,0,0,3,1),(2,0,2,0,0,2,1),\\(2,1,0,0,0,1,2),
(2,1,1,1,0,1,1),(2,1,2,0,0,0,2),(2,2,0,2,0,0,1),(2,2,1,0,1,0,1),\\(3,0,0,1,0,2,1),
(3,0,1,0,0,1,2),(3,0,2,1,0,1,1),(3,0,3,0,0,0,2),(3,1,0,0,1,1,1),\\(3,1,0,1,0,0,2),
(3,1,1,2,0,0,1),(3,1,2,0,1,0,1),(4,0,0,0,0,0,3),(4,0,0,2,0,1,1),\\(4,0,1,0,1,1,1),
(4,0,1,1,0,0,2),(4,1,0,1,1,0,1),(5,0,0,0,1,0,2),(6,0,0,0,2,0,1) 
\bigg\}.
\end{multline*}

The following structural results allows us to exclude the remaining cases.
\begin{proposition}\label{seven}
Let $\mathcal{CL}$ be an arrangement of $k \geq 2$ conics and and $d\geq 1$ lines with singularities no worse than those in Assumption~$(\square)$. Then 
$$t_{5}+t_{7}+d_{8} \leq \binom{k}{2}.$$
\end{proposition}
\begin{proof}
This inequality follows straightforwardly from the fact that a pair of smooth conics can have at most one singularity of type $A_{5}$ or at most one singularity of type $A_{7}$. Furthermore, an $A_{5}$ singularity is a building block of each $D_{8}$ singularity. Since there are $\binom{k}{2}$ pairs of conics, the assertion follows.
\end{proof}
As a consequence, most weak combinatorial types can be excluded. It remains to analyze seven cases, namely:
\begin{multline*}
C(n_{2},n_{3},t_{3},t_{5},t_{7},d_{6},d_{8}) \in \bigg\{
(0,2,0,0,0,2,1),(0,2,2,0,0,1,1),(0,2,4,0,0,0,1), \\
(1,1,1,0,0,2,1),(1,1,3,0,0,1,1),(2,0,0,0,0,3,1), (2,0,2,0,0,2,1)\}.
\end{multline*}
In order to approach these cases, we need to introduce some tools and results that will be crucial for our further considerations.We begin by recalling a result due to Dimca, suitably adapted to our setting; see \cite[Theorem~1.3]{POG}. Let us recall that if $C=\{f=0\}$ is a reduced free plane curve of degree $d$, then the exponents of $C$ are defined as the pair of the form ${\rm exp}(C) = (\delta_{1}, \delta_{2}) = ({\rm mdr}(f), d-1-{\rm mdr}(f)).$
\begin{theorem}[Dimca]
\label{dimm}
Let $C \subset \mathbb{P}^{2}_{\mathbb{C}}$ be a reduced curve of degree $d$ and let $L$ be a line which is not an irreducible component of $C$. Assume that $C' = C \cup L : f'  = 0$ is a free curve with exponents $(\delta_{1}', \delta_{2}')$ and it admits only ${\rm ADE}$-singularities. Then the exponents of $(\delta_{1},\delta_{2})$ (respectively $(\delta_{1},\delta_{2},\delta_{3}))$ of the free curve $C$ (respectively plus-one generated) and the number $r = |C \cap L|$ of intersection points satisfy one of the following conditions:
\begin{enumerate}
    \item $\delta_{1}' < \delta_{2}'$, $\delta_{1} = \delta_{1}'$ and $\delta_{2}=\delta_{2}' - 1$. In this case $C$ is free with $r=\delta_{1}+1$.
    \item $\delta_{1}=\delta_{1}' -1$ and $\delta_{2}=\delta_{2}'$. In this case $C$ is free and $r= \delta_{2}+1$.
    \item $\delta_{1}=\delta_{1}'$,  $\delta_{2}=\delta_{2}'$ and in this situation $C$ is plus-one generated curve with $r=d-\delta_{3}$.
\end{enumerate}
\end{theorem}
The following corollary is straightforward. 
\begin{corollary}
\label{ell}
In the setting of Theorem \ref{dimm}, assume that $C' = C \cup L$ is a maximizing curve of degree $d' = {\rm deg}(C')=2m+1$ for some $m\geq 2$, and let $d = {\rm deg}(C)$. Then either
\begin{enumerate}
    \item $r=m$ and $C$ is free, or
    \item $r \leq m-1$ and $C$ is plus-one generated.
\end{enumerate}
\end{corollary}
\begin{proof}
If $C'$ is maximizing then $C'$ is free with exponents $(\delta_{1}', \delta_{2}') = (m-1, m+1)$, and by Theorem \ref{dimm} there are three cases to consider.
\begin{enumerate}
    \item $C$ is free with exponents $(\delta_{1},\delta_{2})=(m-1,m)$ and $r = m$, which is our first situation.
    \item $C$ is free with exponents $(\delta_{1},\delta_{2}) = (m-2,m+1)$. However, by \cite[Theorem 1.2]{DimSer}, the following inequality always holds:
    $$\delta_{1} > \frac{1}{2}\cdot d-2 = m-2,$$
and we get a contradiction.    
    \item $C$ is plus-one generated with $(m_1,m+1,\delta_{3})$ and $m+1 \leq \delta_{3} \leq 2m-1$. Moreover,
    $$r = d-\delta_{3} = 2m-\delta_{3} \leq 2m-m-1 = m-1.$$
\end{enumerate}
This completes the proof.
\end{proof}
Let us come back to our scenario of maximizing curves of degree seven. For our next result we will need the following definition.
\begin{definition}
Let $C \subset \mathbb{P}^{2}_{\mathbb{C}}$ be a reduced plane curve and let $L$ be a line that is not an irreducible component of $C$. Let $C' = C \cup L$ and $P \in {\rm Sing}(C')$. We define the following 
$$\triangle_{P}( \text{singularity type of } P \in C; \text{ singularity type of } P \in C') = (i_{p}(L,C), \tau_{p}(C') - \tau_{p}(C)). $$
Moreover, if $P$ is a smooth point of $C$ we just write $\triangle({\rm smooth}, \star)$.
\end{definition}
\begin{remark}
Let us briefly explain the above notation. Let $C$ be a smooth conic and $P \in L \cap C$ be a tacnode. For $C$ our point $P$ is smooth, hence
$$\triangle_{P}({\rm smooth}; A_{3}) = (2, 3).$$
\end{remark}
\begin{proposition}
\label{marcin}
Let \( C \subset \mathbb{P}^2_{\mathbb{C}} \) be a reduced plane curve of degree \(6\), 
and let \( L \) be a line which is not an irreducible component of \( C \).
Assume that all irreducible components of \( C \) are smooth.
Set
\[
C' := C \cup L .
\]
Assume that \( C' \) is a maximizing curve and that all its singularities are of the following types:
\[
A_{2k+1} \quad (k \ge 0), \qquad \text{or} \qquad D_j \quad (j \ge 4 \text{ even}).
\]
Then either\begin{enumerate}
    \item $C$ is free with exponents $(\delta_{1},\delta_{2})=(2,3)$ and $r=3$, or
    \item $C$ is plus-one generated with exponents $(2,4,4)$ and $r=2$. Moreover, the two singular points $P_{1}, P_{2} \in L \cap C$ are of type $D_{6}$, and the line $L$ is tangent to an irreducible component $C_{1}$ of $C$ at both $P_{1}$ and $P_{2}$.
\end{enumerate}
\end{proposition}
\begin{proof}
The only non-trivial thing is to rule out the case $r=1$. If $C$ is plus-one generated then
\[\tau(C) = (d-1)^2 - \delta_{1}(d-\delta_{1}-1)-1-(\delta_{3}-\delta_{2}) = 18 -(\delta_{3} - \delta_{2}).\]
Since $C'$ is maximizing we have $\tau(C')=28$ and $\triangle_{C',C}^{\tau}:= \tau(C') - \tau(C) = 10 + (\delta_{3}-\delta_{2})\geq 10.$
According to our assumptions, we have the following possibilities for $\triangle_{P}(\square \,; \star)$ :
\[\begin{tabular}{|c|}
\hline
$\triangle_{P}({\rm smooth}; A_{1}) = (1,1)$\\
$\triangle_{P}({\rm smooth}; A_{3}) = (2,3)$\\
$\triangle_{P}(A_1; D_{4}) = (2,3)$\\
$\triangle_{P}(A_1; D_{6}) = (3,5)$\\
$\triangle_{P}(A_3; D_{4}) = (2,3)$\\
$\triangle_{P}(A_{2k+1}; D_{4k}) = (2,3)$ \text{ with } $k\geq 2$\\
\hline
\end{tabular}\]
Observe that the case $\triangle_{P}(A_{1}; D_{6}) = (3,5)$ can occur only when the line $L$ is tangent to one of the irreducible components of $C$ at an $A_{1}$ singular point. Moreover, since we assumed that all components are smooth, this $A_{1}$ singularity must be an intersection point of two irreducible components of $C$. Since $\triangle_{C',C}^{\tau} \ge 10$, the above considerations imply that $r \ge 2$.
If $r = 2$, then there is only one admissible combination, namely
\[ \triangle_{P_{1}}(A_{1};D_{6}) = \triangle_{P_{2}}(A_{1};D_{6}) = (3,5). \]
In this case, we obtain $\triangle_{C',C}^{\tau} = 10$, which implies that $\delta_{3} = \delta_{2}$.
Consequently, the line $L$ must be tangent to one of the irreducible components of $C$
at both points $P_{1}$ and $P_{2}$.
\end{proof}

Using a computer program we can determine all distinct (up to the permutations of curves and singularities) $\mathrm{PIM}(C)$ for each of the remaining seven admissible weak combinatorics (see Proposition \ref{seven}) that are B\'ezout constructible. We found the PIM($C$) matrices only for four of these cases, namely:
\begin{align*}
C(n_{2},n_{3},t_{3},t_{5},t_{7},d_{6},d_{8}) \in \big\{
(0,2,4,0,0,0,1),
(1,1,1,0,0,2,1), \\(2,0,0,0,0,3,1), (2,0,2,0,0,2,1)\big\}.
\end{align*}

Now we are going to exclude the remaining four cases directly by the so-called Coordinatization method that is well-established in the theory of line arrangements. For the remaining part of this paper, we change the notation used in ${\rm PIM}$ matrices introduced in Definition \ref{def: PIM} and we label each conic by $C_i$ and each line by $L_j$, for clarity.

First, we are going to prove the following useful lemma.

\begin{lemma}\label{lem: bitangent}
If two conics $C_1$ and $C_2$ are tangent with multiplicity $3$, then there exists only one line tangent to both conics at two different points.

\begin{proof}
From \cite[Proposition 4.2.1]{Celal}, we can assume that
\[C_1: x^2 - yz = 0, \quad C_2 : x^2 + by^2 + cxy - yz = 0.\]
for some $b, c \in \mathbb{C}$, $c \neq 0$. The conics dual to $C_1$ and $C_2$ are then given by
\[C_1': x^2 - 4yz = 0, \quad C_2': x^2 + (c^2 - 4b)z^2 + 2cxz - 4yz = 0.\]
In order to find the common tangent lines, we need to solve the system of equations
\[\left\{\begin{array}{l}
x^2 - 4yz = 0, \\
x^2 + (c^2 - 4b)z^2 + 2cxz - 4yz = 0.
\end{array}\right.\]
If $z = 0$, then $x = 0$ and we obtain the solution $[x:y:z] = [0:1:0]$, which is dual to the line $y = 0$. This line is actually tangent to $C_1$ and $C_2$ at the same point $[0 : 0 : 1]$ (the point of tangency of $C_1$ and $C_2$). Therefore we can assume $z = 1$. It suffices to show that under this assumption the above system of equations has only one solution. By subtracting both equations, we get
\[c^2 - 4b + 2cx = 0,\]
hence $x = \frac{4b-c^2}{2c}$ and $y = \frac{x^2}{4} = \frac{(4b-c^2)^2}{16c^2}$. The obtained solution corresponds to a line
\[(32bc - 8c^3)x + (4b-c^2)^2y + 16c^2z = 0.\]
\end{proof}
\end{lemma}

\begin{proposition}
There does not exist an arrangement $\mathcal{C} = C_1 \cup C_2 \cup L_1 \cup L_2 \cup L_3$ of two conics and three lines, having the weak-combinatorial type
\[
\mathcal{C}(n_{2},n_{3},t_{3},t_{5},t_{7},d_{6},d_{8}) \in \{(0,2,4,0,0,0,1), (2,0,2,0,0,2,1)\}.\]
\end{proposition}

\begin{proof}
    Both of these weak-combinatorial types require for the arrangement to have four points in which two curves intersect with multiplicity $2$. Such an intersection can only be realized by a line and a conic (since conics $C_1$ and $C_2$ intersect with multiplicity $3$) and therefore we need to have two distinct lines tangent to both conics at two different points. However, this is not possible by Lemma \ref{lem: bitangent}.
\end{proof}


\begin{proposition}
There does not exist an arrangement $\mathcal{C} = C_1 \cup C_2 \cup L_1 \cup L_2 \cup L_3$ of two conics and three lines, having the weak-combinatorial type
\[
\mathcal{C}(n_{2},n_{3},t_{3},t_{5},t_{7},d_{6},d_{8}) = (2, 0, 0, 0, 0, 3, 1).\]
\end{proposition}

\begin{proof}
    We have only one possible B\'ezout constructible ${\rm PIM}(\mathcal{C})$ matrix in this case:
\[
\begin{blockarray}{ccccccc}
 & A_1 & A_1 & D_6 & D_6 & D_6 & D_8\\
\begin{block}{c|cccccc|}
  C_{1}C_{2} & 0 & 0 & 0 & 0 & 1 & 3 \\
  C_{1}L_{1} & 0 & 0 & 2 & 0 & 0 & 0 \\
  C_{1}L_{2} & 0 & 0 & 0 & 0 & 2 & 0 \\
  C_{1}L_{3} & 0 & 0 & 1 & 0 & 0 & 1 \\
  C_{2}L_{1} & 0 & 0 & 0 & 2 & 0 & 0 \\
  C_{2}L_{2} & 0 & 0 & 0 & 1 & 1 & 0 \\
  C_{2}L_{3} & 0 & 1 & 0 & 0 & 0 & 1 \\
  L_{1}L_{2} & 0 & 0 & 0 & 1 & 0 & 0 \\
  L_{1}L_{3} & 0 & 0 & 1 & 0 & 0 & 0 \\
  L_{2}L_{3} & 1 & 0 & 0 & 0 & 0 & 0 \\
\end{block}
\end{blockarray}.
 \]    
However, in this arrangement, the line $L_3$ intersects the remaining curves in four different points, which contradicts Proposition \ref{marcin}.
\end{proof}

\begin{proposition}
There does not exist an arrangement $\mathcal{C} = C_1 \cup C_2 \cup L_1 \cup L_2 \cup L_3$ of two conics and three lines, having the weak-combinatorial type
\[
\mathcal{C}(n_{2},n_{3},t_{3},t_{5},t_{7},d_{6},d_{8}) = (1, 1, 1, 0, 0, 2, 1).\]
\end{proposition}

\begin{proof}
We can directly check that we have two possible B\'ezout constructible ${\rm PIM}(\mathcal{C})$ matrices in this case. The first one is the following:
\[
\begin{blockarray}{ccccccc}
 & A_1 & D_4 & A_3 & D_6 & D_6 & D_8\\
\begin{block}{c|cccccc|}
  C_{1}C_{2} & 1 & 0 & 0 & 0 & 0 & 3 \\
  C_{1}L_{1} & 0 & 0 & 0 & 0 & 2 & 0 \\
  C_{1}L_{2} & 0 & 0 & 0 & 0 & 1 & 1 \\
  C_{1}L_{3} & 0 & 0 & 2 & 0 & 0 & 0 \\
  C_{2}L_{1} & 0 & 0 & 0 & 2 & 0 & 0 \\
  C_{2}L_{2} & 0 & 1 & 0 & 0 & 0 & 1 \\
  C_{2}L_{3} & 0 & 1 & 0 & 1 & 0 & 0 \\
  L_{1}L_{2} & 0 & 0 & 0 & 0 & 1 & 0 \\
  L_{1}L_{3} & 0 & 0 & 0 & 1 & 0 & 0 \\
  L_{2}L_{3} & 0 & 1 & 0 & 0 & 0 & 0 \\
\end{block}
\end{blockarray}.
 \]
From \cite[Proposition 4.2.1]{Celal}, we can assume that
\[C_1: x^2 - yz = 0, \quad C_2 : x^2 + by^2 + cxy - yz = 0,\]
for some $b, c \in \mathbb{C}$, $c \neq 0$. These conics intersect at points $[0 : 0 : 1]$ (the $D_8$ singularity) and $[-bc : c^2 : b^2]$ (the $A_1$ singularity). The line $L_1$ is tangent to $C_1$ and $C_2$ at two distinct points, therefore by Lemma \ref{lem: bitangent} it is given by
\[L_1: (32bc - 8c^3)x + (4b-c^2)^2y + 16c^2z = 0.\]
Moreover, we can explicitly compute the coordinates of both points of tangency (which are the $D_6$ singularities). The line $L_1$ is tangent to $C_1$ and $C_2$ at points $[16bc - 4c^3 : -16c^2 : -(4b-c^2)^2]$ and $[16bc + 4c^3 : -16c^2 : 3c^4 - 8bc^2 - 16b^2]$, respectively. The line $L_2$ passes through the first point and through the $D_8$ singularity $[0 : 0 : 1]$, therefore it is given by
\[L_2: 4cx + (4b - c^2)y = 0\]
From there, we can compute the coordinates of the $D_4$ singularity, which is the second intersection point of $C_2$ and $L_2$ (besides $[0 : 0 : 1]$). We obtain the coordinates 
\[[-16bc + 4c^3 : 16c^2 : 16b^2 - 8bc^2 + 5c^4].\]
We can now compute the equation of $L_3$. This line is passing through the $D_4$ singularity and one of the $D_6$ singularities (the intersection of $C_2$ and $L_1$). We have
\[L_3: (32bc  -16c^3)x + (16b^2 - 16bc^2 - c^4)y + 16c^2z = 0\]
Finally, the obtained line has to be tangent to $C_1$ (the $A_3$ singularity). Hence the system of equations
\[\left\{\begin{array}{l}
x^2 - yz = 0, \\
(32bc - 16c^3)x + (16b^2 - 16bc^2 - c^4)y + 16c^2z = 0
\end{array}\right.\]
needs to have only one solution. We can assume $y = 1$ (because $y = 0$ implies $x = 0$ and $z = 0$, a contradiction). Hence $z = x^2$ and after substituting it into the second equation, we obtain
\[16c^2x^2 + (32bc - 16c^3)x + (16b^2 - 16bc^2 - c^4) = 0.\]
The determinant of this equation has to be equal to $0$. However, it is equal to $360c^2$ and by our assumptions $c \neq 0$, which is a contradiction.

The second possible ${\rm PIM}(\mathcal{C})$ matrix is

\[
\begin{blockarray}{ccccccc}
 & A_1 & D_4 & A_3 & D_6 & D_6 & D_8\\
\begin{block}{c|cccccc|}
  C_{1}C_{2} & 0 & 0 & 0 & 1 & 0 & 3 \\
  C_{1}L_{1} & 0 & 0 & 0 & 0 & 1 & 1 \\
  C_{1}L_{2} & 0 & 0 & 0 & 0 & 2 & 0 \\
  C_{1}L_{3} & 0 & 0 & 0 & 2 & 0 & 0 \\
  C_{2}L_{1} & 0 & 1 & 0 & 0 & 0 & 1 \\
  C_{2}L_{2} & 0 & 0 & 2 & 0 & 0 & 0 \\
  C_{2}L_{3} & 0 & 1 & 0 & 1 & 0 & 0 \\
  L_{1}L_{2} & 0 & 0 & 0 & 0 & 1 & 0 \\
  L_{1}L_{3} & 0 & 1 & 0 & 0 & 0 & 0 \\
  L_{2}L_{3} & 1 & 0 & 0 & 0 & 0 & 0 \\
\end{block}
\end{blockarray}.
 \]
As before, from \cite[Proposition 4.2.1]{Celal}, we can assume that
\[C_1: x^2 - yz = 0, \quad C_2 : x^2 + by^2 + cxy - yz = 0,\]
for some $b, c \in \mathbb{C}$, $c \neq 0$. These conics intersect at points $[0 : 0 : 1]$ (our $D_8$ singularity) and $[-bc : c^2 : b^2]$ (one of the $D_6$ singularities) with multiplicities $3$ and $1$, respectively. Based on the table above, the line $L_3$ is tangent to $C_1$ at $[-bc : c^2 : b^2]$, therefore
\[L_3 : 2bcx + b^2y + c^2z = 0. \]
The line $L_3$ intersects $C_2$ at two points, one of which is $[-bc : c^2 : b^2]$ and the second one is the $D_4$ singularity. By solving the system of equations
\[\left\{\begin{array}{l}
2bcx + b^2y + c^2z = 0,\\
x^2 + by^2 + cxy - yz = 0,
\end{array}\right.\]
we obtain the coordinates of the $D_4$ singularity, namely $\big[-\frac{b}{c} - c : 1 : \frac{b^2}{c^2} + 2b\big]$. The line $L_1$ passes through this point as well as the $D_8$ singularity $[0 : 0 : 1]$, therefore it is given by
\[L_1 : x + \left(\frac{b}{c} + c\right)y = 0.\]
We can now compute the coordinates of the second $D_6$ singularity, which is one of the intersection points of $C_1$ and $L_1$ (besides the $D_8$ singularity). Similarly as before, we can now compute the coordinates of the second $D_6$ singularity: $\big[-\frac{b}{c}-c : 1 : \frac{b^2}{c^2} + 2b + c^2\big]$. Note that the line $L_2$ is tangent to $C_1$ at this point. We therefore obtain
\[L_2 : 2\left(-\frac{b}{c}-c\right)x + \left(\frac{b^2}{c^2} + 2b + c^2\right)y - z = 0.\]
This line has to be tangent to $C_2$, which means the system of equations
\[\left\{\begin{array}{l}
2\left(-\frac{b}{c}-c\right)x + \left(\frac{b^2}{c^2} + 2b + c^2\right)y - z = 0,\\
x^2 + by^2 + cxy - yz = 0,
\end{array}\right.\]
needs to have exactly one solution. After eliminating $z$ and simplifying, we obtain
\[x^2 + \left(\frac{2b}{c} + 3c\right)xy + \left(\frac{b^2}{c^2} + 3b + c^2\right)y^2 = 0\]
and the determinant of this equation has to be equal to $0$. However, the determinant is equal to $5c^2$ and by our assumptions $c \neq 0$, which is a contradiction.
\end{proof}

We conclude this section with the following proposition, which follows from the preceding results.

\begin{proposition}\label{result2}
There is no maximizing arrangement of two conics and three lines admitting the singularities listed in Assumption~$(\square)$.
\end{proposition}

\paragraph{Scenario 3: one line and three conics.}
According to our list in the Appendix, there are exactly \(77\) weak combinatorics of one conic and five lines. Using Theorem \ref{HP} and Theorem~\ref{main}, we can narrow down the number of cases to consider to a finite subcollection, which we enumerate below. We are left with $32$ cases.
\begin{multline}
\label{lc1}
\mathcal{C}(n_{2},n_{3},t_{3},t_{5},t_{7},d_{6},d_{8}) \in \bigg\{
(0,0,0,0,0,2,2),(0,0,1,1,0,2,1),(0,0,2,0,0,1,2),\\
(0,0,3,1,0,1,1),(0,0,4,0,0,0,2),(0,0,5,1,0,0,1),(0,1,0,0,0,0,3),\\
(0,1,0,2,0,1,1),(0,1,1,0,1,1,1),(0,1,1,1,0,0,2),(0,1,2,2,0,0,1),\\
(0,1,3,0,1,0,1),(0,2,0,1,1,0,1),(1,0,0,0,1,2,1),(1,0,0,1,0,1,2),\\
(1,0,1,0,0,0,3),(1,0,1,2,0,1,1),(1,0,2,0,1,1,1),(1,0,2,1,0,0,2),\\
(1,0,3,2,0,0,1),(1,0,4,0,1,0,1),(1,1,0,0,1,0,2),(1,1,0,3,0,0,1),\\
(1,1,1,1,1,0,1),(2,0,0,1,1,1,1),(2,0,0,2,0,0,2),(2,0,1,0,1,0,2),\\
(2,0,1,3,0,0,1),(2,0,2,1,1,0,1),(2,1,0,0,2,0,1),(3,0,0,2,1,0,1),\\
(3,0,1,0,2,0,1) \bigg\}. 
\end{multline}

As before, we use a computer program to determine for which of these weak combinatorics we can find the corresponding B\'ezout constructible ${\rm PIM}(\mathcal{C})$ matrices. It turns out that such a matrix exists only for the following $12$ cases:
\begin{multline*}
\mathcal{C}(n_{2},n_{3},t_{3},t_{5},t_{7},d_{6},d_{8}) \in \bigg\{
(0,0,0,0,0,2,2),(0,0,1,1,0,2,1),(0,0,4,0,0,0,2),\\
(0,1,0,0,0,0,3),(0,1,3,0,1,0,1),(1,0,0,1,0,1,2),(1,0,1,2,0,1,1),\\
(1,0,3,2,0,0,1),(1,1,1,1,1,0,1),(2,0,0,1,1,1,1),(2,1,0,0,2,0,1),\\
(3,0,1,0,2,0,1) \bigg\}. 
\end{multline*}

In the following propositions, we prove that none of these weak combinatorics can be realized over $\mathbb{C}$.

\begin{proposition}
There does not exist an arrangement $\mathcal{C} = C_1 \cup C_2 \cup C_3 \cup L$ of one line and three conics with a weak combinatorics
\begin{align*}
C(n_{2},n_{3},t_{3},t_{5},t_{7},d_{6},d_{8}) \in \big\{
(0,0,0,0,0,2,2),
(0,1,3,0,1,0,1),\\ (2,1,0,0,2,0,1), (3,0,1,0,2,0,1) \big\}.
\end{align*}
\end{proposition}
\begin{proof}
    Each of these weak combinatorics have only one possible B\'ezout constructible ${\rm PIM}(\mathcal{C})$ matrix and in each case we have $\#((C_1 \cup C_2 \cup C_3) \cap L) = 3$. Therefore by Proposition \ref{marcin}, after deleting the line $L_1$ we should obtain a free arrangement of three conics. However, it can be checked that in each case the obtained arrangement is not on the list of all possible free arrangements of three conics presented in \cite[Theorem A]{MZZ}.
\end{proof}
\begin{lemma}\label{lem: a5}
    If two distinct smooth conics $C_1$ and $C_2$ intersect at singularities $A_5$ and $A_1$, then the equation of $C_2$ can be written in the form
    \[
    C_2: \lambda C_1 + \ell_1\ell_2 = 0,
    \]
    where $\lambda \in \mathbb{C} \setminus \{0\}$, $\ell_1$ is the equation of a line tangent to $Q_1$ at the $A_5$ singularity and $\ell_2$ is the equation of a line passing through both singularities.
\end{lemma}
\begin{proposition}
There does not exist an arrangement $\mathcal{C} = C_1 \cup C_2 \cup C_3 \cup L$ of one line and three conics with a weak combinatorics
\[
\mathcal{C}(n_{2},n_{3},t_{3},t_{5},t_{7},d_{6},d_{8}) = (0, 0, 1, 1, 0, 2, 1).\]
\end{proposition}

\begin{proof}
    The following ${\rm PIM}(\mathcal{C})$ matrix is the only admissible one which is B\'ezout constructible:
    \[
    \begin{blockarray}{cccccc}
     & A_3 & A_5 & D_6 & D_6 & D_8 \\
    \begin{block}{c|ccccc|}
      C_{1}C_{2} & 0 & 3 & 0 & 0 & 1 \\
      C_{1}C_{3} & 0 & 0 & 0 & 1 & 3 \\
      C_{1}L & 0 & 0 & 0 & 2 & 0 \\
      C_{2}C_{3} & 2 & 0 & 1 & 0 & 1 \\
      C_{2}L & 0 & 0 & 2 & 0 & 0 \\
      C_{3}L & 0 & 0 & 1 & 1 & 0 \\
    \end{block}
    \end{blockarray}
    \]
    From \cite[Proposition 4.2.1]{Celal}, we can assume that
    \[
    C_1 : x^2 - yz = 0, \quad C_3 : x^2 + by^2 + cxy - yz = 0,
    \]
    with $b, c \in \mathbb{C}$ and $c \neq 0$. These conics intersect at points $[0 : 0 : 1]$ (our $D_8$ singularity) and $[-bc : c^2 : b^2]$ (one of the $D_6$ singularities) with multiplicities $3$ and $1$, respectively. The $A_5$ singularity lies on $C_1$ and therefore it has coordinates $[p : 1 : p^2]$ for some $p \in \mathbb{C}$, $p \neq \frac{-b}{c}$. From Lemma \ref{lem: a5}, the equation of $C_2$ can be written as
    \begin{align*}
    C_2 &:\lambda(x^2 - yz) + (x - py)(2px - p^2y - z) \\
    &= (\lambda + 2p)x^2 + p^3y^2 - 3p^2xy - xz + (p - \lambda)yz = 0.
    \end{align*}
    for some $\lambda \in \mathbb{C}\setminus\{0\}.$ Based on the ${\rm PIM}(\mathcal{C})$ matrix, the line $L$ is tangent to $C_1$ at $[-bc : c^2 : b^2]$ and therefore it is given by $2bcx + b^2y + c^2z = 0$. The second $D_6$ singularity is one of the intersection points of $C_2$ and $C_3$ and the line is tangent to $C_2$ at this point. Therefore the system of equations
    \[
    \left\{\begin{array}{ll}
    C_2: & (\lambda + 2p)x^2 + p^3y^2 - 3p^2xy - xz + (p - \lambda)yz = 0, \\
    L: & 2bcx + b^2y + c^2z = 0
    \end{array}\right.
    \]
    needs to have exactly one solution. From there, we obtain
    \begin{equation}\label{quadr}
    (\lambda + 2p + 2q)x^2 + (2\lambda q - 3p^2 - 2pq + q^2)xy + (\lambda q^2 + p^3 - pq^2)y^2 = 0,
    \end{equation}
    where $q = \frac{b}{c}$. The determinant of the above equation is equal to
    \[
    (p + q)^3(p + q - 4\lambda).
    \]
    and since $L$ has to be tangent to $C_2$, this determinant has to be equal to $0$. Note that $p + q = 0$ implies $p = -\frac{b}{c}$, contradicting our assumptions. Therefore we have $\lambda = \frac{1}{4}(p + q)$. After factorization, the equation \eqref{quadr} reduces to
    \[
    \frac{1}{4}(p + q)\big(3x - (2p-q)y\big)^2 = 0,
    \]
    hence $3x - (2p-q)y = 0$ and therefore the coordinates of the second $D_6$ singularity are $[2p - q : 3 : -4pq-q^2]$, which can be rewritten as $[2pc^2 - bc : 3c^2 : -4pbc - b^2]$. Since this point has to also lie on $C_3$, we obtain yet another equation:
    \[
    (2pc^2 - bc)^2 + b(3c^2)^2 + c(2pc^2 - bc)(3c^2) - 3c^2(-4pbc-b^2) = 0,
    \]
    which factorizes into $2c^2(pc + b)(3c^2 + 2cp + 2b) = 0$. Since $c \neq 0$ and $p \neq -\frac{b}{c}$, we obtain $b = \frac{-3c^2-2cp}{2}$. As a result, the equation of $C_2$ can be rewritten as
    \[
    C_2: (16p-3c)x^2 + 8p^3y^2 - 24p^2xy - 8xz + (3c + 8p)yz = 0.
    \]
    The only condition left for this arrangement to exist is that $C_2$ is tangent to $C_3$ at some point, the remaining $A_3$ singularity. Since $C_3$ can be parametrized as
    \[
    C_3 = \{[0 : 0 : 1]\} \cup \{[u : 1 : u^2 + b + cu] : u \in \mathbb{C}\},
    \]
    we can substitute this parametrization to the equation of $C_2$ ($x = u$, $y = 1$, $z = u^2 + b + cu$) to determine the intersection points of $C_2$ and $C_3$. Using the equality $b = \frac{-3c^2-2cp}{2}$, we obtain
    \[
    -8u^3+(24p-8c)u^2+(15c^2+16cp-24p^2)u+\frac{-9c^3-30c^2p-16cp^2+16p^3}{2}=0,
    \]
    which can be factorized as
    \[
    \frac{1}{2}(c+2p-2u)\big(8u^2+(12c-16p)u-9c^2-12cp+8p^2\big) = 0.
    \]
    In order to have the point of tangency, the determinant of the second factor has to be equal to $0$. However, it is equal to $432c^2$ and by our assumptions $c \neq 0$, which gives a contradiction.
\end{proof}

\begin{proposition}
There does not exist an arrangement $\mathcal{C} = C_1 \cup C_2 \cup C_3 \cup L$ of three conics and one line, having the weak-combinatorial type
\[
\mathcal{C}(n_{2},n_{3},t_{3},t_{5},t_{7},d_{6},d_{8}) = (0, 0, 4, 0, 0, 0, 2).\]
\end{proposition}

\begin{proof}
The following ${\rm PIM}(\mathcal{C})$ matrix is the only admissible one which is B\'ezout constructible:
\[
\begin{blockarray}{ccccccc}
 & A_3 & A_3 & A_3 & A_3 & D_8 & D_8 \\
\begin{block}{c|cccccc|}
  C_{1}C_{2} & 0 & 0 & 0 & 0 & 1 & 3 \\
  C_{1}C_{3} & 0 & 0 & 0 & 0 & 3 & 1 \\
  C_{1}L & 0 & 0 & 2 & 0 & 0 & 0 \\
  C_{2}C_{3} & 0 & 0 & 0 & 2 & 1 & 1 \\
  C_{2}L & 2 & 0 & 0 & 0 & 0 & 0 \\
  C_{3}L & 0 & 2 & 0 & 0 & 0 & 0 \\
\end{block}
\end{blockarray}
\]
By deleting $L$ from $\mathcal{C}$, we obtain a free arrangement of three conics $C_1$, $C_2$ and $C_3$, with one $A_3$ singularity and two $D_8$ singularities. From \cite[Proposition 3.20]{MZZ}, we may assume that 
\begin{align*}
    &C_1: x^2 - yz = 0, \\
    &C_2: x^2 - 2xy - yz = 0, \\
    &C_3: x^2 + 2xz - yz = 0.
\end{align*}
The line $L$ has to be tangent to all of these three conics. From the proof of Lemma \ref{lem: bitangent} for $C_1$ and $C_2$ (with $b = 0$ and $c = -2$) we can deduce that
\[L: 4x + y + 4z = 0.\]
However, this line is not tangent to $C_3$.
\end{proof}

\begin{proposition}
There does not exist an arrangement $\mathcal{C} = C_1 \cup C_2 \cup C_3 \cup L$ of three conics and one line, having the weak-combinatorial type
\[
\mathcal{C}(n_{2},n_{3},t_{3},t_{5},t_{7},d_{6},d_{8}) = (0, 1, 0, 0, 0, 0, 3).\]
\end{proposition}

\begin{proof}
We have two possible B\'ezout constructible ${\rm PIM}(\mathcal{C})$ matrices. The first one is the following:
\[
\begin{blockarray}{ccccc}
 & D_4 & D_8 & D_8 & D_8 \\
\begin{block}{c|cccc|}
  C_{1}C_{2} & 0 & 3 & 0 & 1 \\
  C_{1}C_{3} & 0 & 0 & 3 & 1 \\
  C_{1}L & 0 & 1 & 1 & 0 \\
  C_{2}C_{3} & 1 & 0 & 0 & 3 \\
  C_{2}L & 1 & 1 & 0 & 0 \\
  C_{3}L & 1 & 0 & 1 & 0 \\
\end{block}
\end{blockarray}
\]
By deleting $L$ from $\mathcal{C}$, we obtain a free arrangement of three conics with one $A_1$ singularity, two $A_5$ singularities and one $D_8$ singularity. From \cite[Proposition 3.24]{MZZ} we may assume that
\begin{align*}
    &C_1: 12x^2 + (i\sqrt{3}+3)xz - 12yz = 0,\\
    &C_2: x^2 - yz = 0, \\
    &C_3: x^2 + 18(i\sqrt{3}+3)y^2 - 18xy - yz = 0.
\end{align*}
The line $L$ needs to pass through the $A_1$ singularity and two $A_5$ singularities. From the above equations it can be deduced that the coordinates of these three points are
\[\big[i\sqrt{3}+3 : 1 : 6i\sqrt{3}+6\big], \; [0 : 1 : 0] \text{ and } \big[6 : 1 : 18i\sqrt{3} - 18\big].\]
However, these points are not collinear.

The second possible ${\rm PIM}(\mathcal{C})$ matrix is
\[
\begin{blockarray}{ccccc}
 & D_4 & D_8 & D_8 & D_8 \\
\begin{block}{c|cccc|}
  C_{1}C_{2} & 1 & 0 & 3 & 0 \\
  C_{1}C_{3} & 1 & 0 & 0 & 3 \\
  C_{1}L & 0 & 0 & 1 & 1 \\
  C_{2}C_{3} & 1 & 3 & 0 & 0 \\
  C_{2}L & 0 & 1 & 1 & 0 \\
  C_{3}L & 0 & 1 & 0 & 1 \\
\end{block}
\end{blockarray}
\]
By deleting $L$ from $\mathcal{C}$, we obtain a free arrangement of three conics with one $D_4$ singularity and three $A_5$ singularities. From \cite[Proposition 3.17]{MZZ}, we may assume that
\begin{align*}
    &C_1: x^2 - yz = 0, \\
    &C_2: 3x^2 + y^2 + 3xy - 3yz = 0, \\
    &C_3: x^2 + 27z^2 + 9xz - yz = 0. 
\end{align*}
The coordinates of three $A_5$ singularities are
\[[0 : 0 : 1], \; [0 : 1 : 0] \text{ and } [-6 : 9 : 1].\]
All three points has to lie on $L$. However, they are not collinear.
\end{proof}

\begin{proposition}
There does not exist an arrangement $\mathcal{C} = C_1 \cup C_2 \cup C_3 \cup L$ of three conics and one line, having the weak-combinatorial type
\[
\mathcal{C}(n_{2},n_{3},t_{3},t_{5},t_{7},d_{6},d_{8}) = (1, 0, 0, 1, 0, 1, 2).\]
\end{proposition}

\begin{proof}
The following ${\rm PIM}(\mathcal{C})$ matrix is the only admissible one which is B\'ezout constructible:
\[
\begin{blockarray}{cccccc}
 & A_1 & A_5 & D_6 & D_8 & D_8 \\
\begin{block}{c|ccccc|}
  C_{1}C_{2} & 0 & 0 & 0 & 3 & 1 \\
  C_{1}C_{3} & 0 & 0 & 1 & 0 & 3 \\
  C_{1}L & 0 & 0 & 1 & 1 & 0 \\
  C_{2}C_{3} & 0 & 3 & 0 & 0 & 1 \\
  C_{2}L & 1 & 0 & 0 & 1 & 0 \\
  C_{3}L & 0 & 0 & 2 & 0 & 0 \\
\end{block}
\end{blockarray}
\]
By deleting $L$ from $\mathcal{C}$, we obtain a free arrangement of three conics with one $A_1$ singularity, two $A_5$ singularities and one $D_8$ singularity. From \cite[Proposition 3.24]{MZZ} we may assume that
\begin{align*}
    &C_1: x^2 + 18(i\sqrt{3}+3)y^2 - 18xy - yz = 0, \\
    &C_2: 12x^2 + (i\sqrt{3}+3)xz - 12yz = 0, \\
    &C_3: x^2 - yz = 0.
\end{align*}
The conics $C_1$ and $C_3$ intersect at points $[0 : 0 : 1]$ and $\big[i\sqrt{3}+3 : 1 : 6i\sqrt{3}+6\big]$, with respective multiplicities $3$ and $1$. Based on the ${\rm PIM}(\mathcal{C})$ matrix, the line $L$ has to be tangent to $C_3$ at $\big[i\sqrt{3}+3 : 1 : 6i\sqrt{3}+6\big]$, hence this line is given by
\begin{equation}\label{arr6_c4}L : 2(i\sqrt{3} + 3)x - (6i\sqrt{3} + 6)y - z = 0.
\end{equation}
Moreover, according to the ${\rm PIM}(\mathcal{C})$ matrix, the line $L$ also passes through the point of tangency of $C_1$ and $C_2$. This point has coordinates
\[\big[6 : 1 : 18i\sqrt{3} - 18\big],\]
but it does not lie on the line given by \eqref{arr6_c4}.
\end{proof}

\begin{proposition}
There does not exist an arrangement $\mathcal{C} = C_1 \cup C_2 \cup C_3 \cup L$ of three conics and one line, having the weak-combinatorial type
\[
\mathcal{C}(n_{2},n_{3},t_{3},t_{5},t_{7},d_{6},d_{8}) = (1, 0, 1, 2, 0, 1, 1).\]
\end{proposition}

\begin{proof}
The following ${\rm PIM}(\mathcal{C})$ matrix is the only admissible one which is B\'ezout constructible:
\[
\begin{blockarray}{ccccccc}
 & A_1 & A_3 & A_5 & A_5 & D_6 & D_8 \\
\begin{block}{c|cccccc|}
  C_{1}C_{2} & 0 & 0 & 0 & 0 & 1 & 3 \\
  C_{1}C_{3} & 0 & 0 & 0 & 3 & 0 & 1 \\
  C_{1}L & 0 & 0 & 0 & 0 & 2 & 0 \\
  C_{2}C_{3} & 0 & 0 & 3 & 0 & 0 & 1 \\
  C_{2}L & 1 & 0 & 0 & 0 & 1 & 0 \\
  C_{3}L & 0 & 2 & 0 & 0 & 0 & 0 \\
\end{block}
\end{blockarray}
\]
By deleting $L$ from $\mathcal{C}$, we obtain a free arrangement of three conics with one $A_1$ singularity, two $A_5$ singularities and one $D_8$ singularity. From \cite[Proposition 3.24]{MZZ} we may assume that
\begin{align*}
    &C_1: x^2 - yz = 0, \\
    &C_2: x^2 + 18(i\sqrt{3} + 3)y^2 - 18xy - yz = 0, \\
    &C_3: 12x^2 + (i\sqrt{3} + 3)xz - 12yz = 0.
\end{align*}
According to the ${\rm PIM}(\mathcal{C})$, the line $L$ is tangent to $C_1$ at a point, where $C_1$ and $C_2$ intersect transversally. This point has coordinates
\[\big[i\sqrt{3} + 3 : 1 : 6i\sqrt{3} + 6\big]\]
and therefore the line $L$ is given by
\[L: 2(i\sqrt{3} + 3)x - (6i\sqrt{3} + 6)y - z = 0.\]
However, the obtained line is not tangent to $C_3$, which is a contradiction with the ${\rm PIM}(\mathcal{C})$ matrix.
\end{proof}

\begin{proposition}
There does not exist an arrangement $\mathcal{C} = C_1 \cup C_2 \cup C_3 \cup L$ of three conics and one line, having the weak-combinatorial type
\[
C(n_{2},n_{3},t_{3},t_{5},t_{7},d_{6},d_{8}) = (1, 0, 3, 2, 0, 0, 1).\]
\end{proposition}

\begin{proof}
The following ${\rm PIM}(\mathcal{C})$ matrix is the only admissible one which is B\'ezout constructible:
\[
\begin{blockarray}{cccccccc}
 & A_1 & A_3 & A_3 & A_3 & A_5 & A_5 & D_8 \\
\begin{block}{c|ccccccc|}
  C_{1}C_{2} & 0 & 0 & 0 & 0 & 3 & 0 & 1 \\
  C_{1}C_{3} & 0 & 0 & 0 & 0 & 0 & 3 & 1 \\
  C_{1}L & 0 & 0 & 0 & 2 & 0 & 0 & 0 \\
  C_{2}C_{3} & 1 & 0 & 0 & 0 & 0 & 0 & 3 \\
  C_{2}L & 0 & 2 & 0 & 0 & 0 & 0 & 0 \\
  C_{3}L & 0 & 0 & 2 & 0 & 0 & 0 & 0 \\
\end{block}
\end{blockarray}
\]
By deleting $L$ from $\mathcal{C}$, we obtain a free arrangement of three conics with one $A_1$ singularity, two $A_5$ singularities and one $D_8$ singularity. From \cite[Proposition 3.23]{MZZ} we may assume that
\begin{align*}
    &C_1: 12x^2 + (i\sqrt{3} + 3)xz - 12yz = 0, \\
    &C_2: x^2 - yz = 0, \\
    &C_3: x^2 + 18(i\sqrt{3} + 3)y^2 - 18xy - yz = 0.
\end{align*}
Based on the ${\rm PIM}(\mathcal{C})$ matrix, the line $L$ has to be tangent to all three conics, at three distinct points. From the proof of Lemma \ref{lem: bitangent} (with $b = 18(i\sqrt{3}+3)$ and $c = -18$), the unique line tangent to $C_2$ and $C_3$ at two different points is given by
\[84x + (63 - 42i\sqrt{3})y + (8i\sqrt{3} + 12)z = 0.\]
However, it can be checked that this line is not tangent to $C_1$.
\end{proof}

\begin{proposition}
There does not exist an arrangement $\mathcal{C} = C_1 \cup C_2 \cup C_3 \cup L$ of three conics and one line, having the weak-combinatorial type
\[
C(n_{2},n_{3},t_{3},t_{5},t_{7},d_{6},d_{8}) = (1, 1, 1, 1, 1, 0, 1).\]
\end{proposition}

\begin{proof}
The following ${\rm PIM}(\mathcal{C})$ matrix is the only admissible one which is B\'ezout constructible:
\[
\begin{blockarray}{ccccccc}
 & A_1 & D_4 & A_3 & A_5 & A_7 & D_8 \\
\begin{block}{c|cccccc|}
  C_{1}C_{2} & 0 & 1 & 0 & 0 & 0 & 3 \\
  C_{1}C_{3} & 0 & 0 & 0 & 0 & 4 & 0 \\
  C_{1}L & 0 & 1 & 0 & 0 & 0 & 1 \\
  C_{2}C_{3} & 1 & 0 & 0 & 3 & 0 & 0 \\
  C_{2}L & 0 & 1 & 0 & 0 & 0 & 1 \\
  C_{3}L & 0 & 0 & 2 & 0 & 0 & 0 \\
\end{block}
\end{blockarray}
\]
By deleting the line $L$ from $\mathcal{C}$, we obtain a free arrangement of three conics with two $A_1$ singularities, two $A_5$ singularities and one $A_7$ singularity. From \cite[Proposition 3.18]{MZZ} we may assume that
\begin{align*}
    &C_1: x^2 - yz = 0, \\
    &C_2: 8x^2 - 3(i\sqrt{3}+1)y^2 + 8i\sqrt{3}xy - 8yz = 0, \\
    &C_3: x^2 + i\sqrt{3}z^2 - yz = 0.
\end{align*}
From the ${\rm PIM}(\mathcal{C})$ matrix, the line $L$ has to pass through both points of intersection of $C_1$ and $C_2$. These two points have coordinates $[0 : 0 : 1]$ and $\big[2i\sqrt{3}+6 : 8i\sqrt{3} + 8 : 3 \big]$ and therefore the line $L$ is given by
\[L : (8i\sqrt{3}+8)x - (2i\sqrt{3} + 6)y = 0.\]
However, this line is not tangent to $C_3$, which is a contradiction with the ${\rm PIM}(\mathcal{C})$ matrix.
\end{proof}

\begin{proposition}
There does not exist an arrangement $\mathcal{C} = C_1 \cup C_2 \cup C_3 \cup L$ of three conics and one line, having the weak-combinatorial type
\[
C(n_{2},n_{3},t_{3},t_{5},t_{7},d_{6},d_{8}) = (2,0,0,1,1,1,1).\]
\end{proposition}

\begin{proof}
The following ${\rm PIM}(\mathcal{C})$ matrix is the only admissible one which is B\'ezout constructible:
\[
\begin{blockarray}{ccccccc}
 & A_1 & A_1 & A_5 & A_7 & D_6 & D_8 \\
\begin{block}{c|cccccc|}
  C_{1}C_{2} & 0 & 0 & 0 & 4 & 0 & 0 \\
  C_{1}C_{3} & 0 & 0 & 3 & 0 & 1 & 0 \\
  C_{1}L & 0 & 0 & 0 & 0 & 2 & 0 \\
  C_{2}C_{3} & 0 & 1 & 0 & 0 & 0 & 3 \\
  C_{2}L & 1 & 0 & 0 & 0 & 0 & 1 \\
  C_{3}L & 0 & 0 & 0 & 0 & 1 & 1 \\
\end{block}
\end{blockarray}
\]
By deleting the line $L$ from $\mathcal{C}$, we obtain a free arrangement of three conics with two $A_1$ singularities, two $A_5$ singularities and one $A_7$ singularity. From \cite[Proposition 3.18]{MZZ} we may assume that
\begin{align*}
    &C_1: x^2 - yz = 0, \\
    &C_2: x^2 + i\sqrt{3}z^2 - yz = 0, \\
    &C_3: 8x^2 - 3(i\sqrt{3}+1)y^2 + 8i\sqrt{3}xy - 8yz = 0.
\end{align*}
Based on the ${\rm PIM}(\mathcal{C})$ matrix, the line $L$ is tangent to $C_1$ at the point in which $C_1$ intersects $C_3$ with multiplicity $1$. This point has coordinates $\big[2i\sqrt{3} + 6 : 8i\sqrt{3} + 8 : 3\big]$ and therefore the line $L$ is given by
\begin{equation}\label{arr10_c4}
L : (4i\sqrt{3} + 12)x - 3y - (8i\sqrt{3} + 8)z = 0.
\end{equation}
However, from the ${\rm PIM}(\mathcal{C})$ matrix, the line $L$ passes through the point in which $C_2$ and $C_3$ intersect with multiplicity $3$, i.e.\ the point $\big[1 : u\sqrt{3} + 1 : 1\big]$. It can be checked that the line given by \eqref{arr10_c4} does not pass through this point.
\end{proof}

Combining the results proved above, we arrive at the following proposition.

\begin{proposition}\label{result3}
There is no maximizing arrangement of three conics and one line admitting the singularities listed in Assumption~$(\square)$.
\end{proposition}

Based on the results presented in Propositions \ref{result1}, \ref{result2} and \ref{result3}, we can now state the main theorem of this paper.

\begin{mtheorem}\label{maintheorem}
There is no conic-line arrangement $C \subset \mathbb{P}^2_{\mathbb{C}}$ of degree seven admitting $A_t$ singularities with $t \in \{1,3,5,7\}$ and $D_l$ singularities with $l \in \{4,6,8\}$. 
\end{mtheorem}

\section*{Acknowledgements}

We would like to thank Piotr Pokora for valuable suggestions and insightful comments, as well as for their support throughout the preparation of this paper.

\vspace{20pt}
Affiliation of the authors: Department of Mathematics, University of the National Education Commission, Krakow, Podchor\k{a}\.{z}ych 2, PL-30-084 Krak\'{o}w, Poland.

\medskip

\L{}ukasz Merta: \texttt{lukasz.merta@uken.krakow.pl}

Marcin Zieli\'nski: \texttt{marcin.zielinski@uken.krakow.pl}

\newpage
\section*{Appendix: Enumeration of all admissible weak-combinatorics}
\paragraph{One conic and five lines.}
We have $103$ possible weak-combinatorial types:
\begin{multline*}
C(n_{2},n_{3},t_{3},t_{5},t_{7},d_{6},d_{8}) \in
\bigg\{ (0,4,0,0,0,2,0),(0,4,2,0,0,1,0),(0,4,4,0,0,0,0), \\ 
(0,5,0,0,0,0,1), (0,5,1,1,0,0,0), (1,3,1,0,0,2,0), (1,3,3,0,0,1,0), (1,3,5,0,0,0,0), \\ 
(1,4,0,1,0,1,0), (1,4,1,0,0,0,1), (1,4,2,1,0,0,0), (1,5,0,0,1,0,0),(2,2,0,0,0,3,0), \\ 
(2,2,2,0,0,2,0), (2,2,4,0,0,1,0), (2,2,6,0,0,0,0), (2,3,0,0,0,1,1), (2,3,1,1,0,1,0), \\ 
(2,3,2,0,0,0,1), (2,3,3,1,0,0,0), (2,4,0,2,0,0,0), (2,4,1,0,1,0,0),(3,1,1,0,0,3,0), \\ 
(3,1,3,0,0,2,0),(3,1,5,0,0,1,0),(3,1,7,0,0,0,0),(3,2,0,1,0,2,0),(3,2,1,0,0,1,1), \\ 
(3,2,2,1,0,1,0),(3,2,3,0,0,0,1),(3,2,4,1,0,0,0),(3,3,0,0,1,1,0),(3,3,0,1,0,0,1), \\ 
(3,3,1,2,0,0,0),(3,3,2,0,1,0,0),(4,0,0,0,0,4,0),(4,0,2,0,0,3,0),(4,0,4,0,0,2,0), \\ 
(4,0,6,0,0,1,0),(4,0,8,0,0,0,0),(4,1,0,0,0,2,1),(4,1,1,1,0,2,0),(4,1,2,0,0,1,1), \\ 
(4,1,3,1,0,1,0),(4,1,4,0,0,0,1),(4,1,5,1,0,0,0),(4,2,0,0,0,0,2),(4,2,0,2,0,1,0), \\ 
(4,2,1,0,1,1,0),(4,2,1,1,0,0,1),(4,2,2,2,0,0,0),(4,2,3,0,1,0,0),(4,3,0,1,1,0,0), \\ 
(5,0,0,1,0,3,0),(5,0,1,0,0,2,1),(5,0,2,1,0,2,0),(5,0,3,0,0,1,1),(5,0,4,1,0,1,0), \\
 (5,0,5,0,0,0,1),(5,0,6,1,0,0,0),(5,1,0,0,1,2,0),(5,1,0,1,0,1,1),(5,1,1,0,0,0,2), \\
 (5,1,1,2,0,1,0),(5,1,2,0,1,1,0),(5,1,2,1,0,0,1),(5,1,3,2,0,0,0),(5,1,4,0,1,0,0), \\ 
 (5,2,0,0,1,0,1),(5,2,0,3,0,0,0),(5,2,1,1,1,0,0),(6,0,0,0,0,1,2),(6,0,0,2,0,2,0), \\ 
 (6,0,1,0,1,2,0),(6,0,1,1,0,1,1),(6,0,2,0,0,0,2),(6,0,2,2,0,1,0),(6,0,3,0,1,1,0), \\ 
 (6,0,3,1,0,0,1),(6,0,4,2,0,0,0),(6,0,5,0,1,0,0),(6,1,0,1,1,1,0),(6,1,0,2,0,0,1), \\ 
 (6,1,1,0,1,0,1),(6,1,1,3,0,0,0),(6,1,2,1,1,0,0),(6,2,0,0,2,0,0),(7,0,0,0,1,1,1), \\ 
 (7,0,0,1,0,0,2),(7,0,0,3,0,1,0),(7,0,1,1,1,1,0),(7,0,1,2,0,0,1),(7,0,2,0,1,0,1), \\ 
 (7,0,2,3,0,0,0),(7,0,3,1,1,0,0),(7,1,0,2,1,0,0),(7,1,1,0,2,0,0),(8,0,0,0,2,1,0), \\ 
 (8,0,0,1,1,0,1),(8,0,0,4,0,0,0),(8,0,1,2,1,0,0),(8,0,2,0,2,0,0),(9,0,0,1,2,0,0) \bigg\}.  
\end{multline*}

\paragraph{Two conics and three lines.}

We have $116$ possible weak-combinatorial types:
\begin{multline*}
C(n_{2},n_{3},t_{3},t_{5},t_{7},d_{6},d_{8}) \in \bigg\{ (0,1,0,0,0,4,0),(0,1,2,0,0,3,0),(0,1,4,0,0,2,0), \\(0,1,6,0,0,1,0),(0,1,8,0,0,0,0),(0,2,0,0,0,2,1),(0,2,1,1,0,2,0), (0,2,2,0,0,1,1), \\
(0,2,3,1,0,1,0),(0,2,4,0,0,0,1), (0,2,5,1,0,0,0), 
(0,3,0,0,0,0,2),(0,3,0,2,0,1,0), \\
(0,3,1,0,1,1,0),(0,3,1,1,0,0,1),(0,3,2,2,0,0,0),(0,3,3,0,1,0,0),(0,4,0,1,1,0,0), \\
(1,0,1,0,0,4,0),(1,0,3,0,0,3,0),(1,0,5,0,0,2,0),(1,0,7,0,0,1,0),(1,0,9,0,0,0,0), \\
(1,1,0,1,0,3,0),(1,1,1,0,0,2,1),(1,1,2,1,0,2,0),(1,1,3,0,0,1,1),(1,1,4,1,0,1,0), \\
(1,1,5,0,0,0,1),(1,1,6,1,0,0,0),(1,2,0,0,1,2,0),(1,2,0,1,0,1,1),(1,2,1,0,0,0,2), \\
(1,2,1,2,0,1,0),(1,2,2,0,1,1,0),(1,2,2,1,0,0,1),(1,2,3,2,0,0,0),(1,2,4,0,1,0,0), \\
(1,3,0,0,1,0,1),(1,3,0,3,0,0,0),(1,3,1,1,1,0,0),(2,0,0,0,0,3,1),(2,0,1,1,0,3,0), \\
(2,0,2,0,0,2,1),(2,0,3,1,0,2,0),(2,0,4,0,0,1,1),(2,0,5,1,0,1,0),(2,0,6,0,0,0,1), \\
(2,0,7,1,0,0,0),(2,1,0,0,0,1,2),(2,1,0,2,0,2,0),(2,1,1,0,1,2,0),(2,1,1,1,0,1,1),\\
(2,1,2,0,0,0,2),(2,1,2,2,0,1,0),(2,1,3,0,1,1,0),(2,1,3,1,0,0,1),(2,1,4,2,0,0,0),\\
(2,1,5,0,1,0,0),(2,2,0,1,1,1,0),(2,2,0,2,0,0,1),(2,2,1,0,1,0,1),(2,2,1,3,0,0,0),\\
(2,2,2,1,1,0,0),(2,3,0,0,2,0,0),(3,0,0,0,1,3,0),(3,0,0,1,0,2,1),(3,0,1,0,0,1,2),\\
(3,0,1,2,0,2,0),(3,0,2,0,1,2,0),(3,0,2,1,0,1,1),(3,0,3,0,0,0,2),(3,0,3,2,0,1,0),\\
(3,0,4,0,1,1,0),(3,0,4,1,0,0,1),(3,0,5,2,0,0,0),(3,0,6,0,1,0,0),(3,1,0,0,1,1,1),\\
(3,1,0,1,0,0,2),(3,1,0,3,0,1,0),(3,1,1,1,1,1,0),(3,1,1,2,0,0,1),(3,1,2,0,1,0,1),\\
(3,1,2,3,0,0,0),(3,1,3,1,1,0,0),(3,2,0,2,1,0,0),(3,2,1,0,2,0,0),(4,0,0,0,0,0,3),\\
(4,0,0,1,1,2,0),(4,0,0,2,0,1,1),(4,0,1,0,1,1,1),(4,0,1,1,0,0,2),(4,0,1,3,0,1,0),\\
(4,0,2,1,1,1,0),(4,0,2,2,0,0,1),(4,0,3,0,1,0,1),(4,0,3,3,0,0,0),(4,0,4,1,1,0,0),\\
(4,1,0,0,2,1,0),(4,1,0,1,1,0,1),(4,1,0,4,0,0,0),(4,1,1,2,1,0,0),(4,1,2,0,2,0,0),\\
(5,0,0,0,1,0,2),(5,0,0,2,1,1,0),(5,0,0,3,0,0,1),(5,0,1,0,2,1,0),(5,0,1,1,1,0,1),\\
(5,0,1,4,0,0,0),(5,0,2,2,1,0,0),(5,0,3,0,2,0,0),(5,1,0,1,2,0,0),(6,0,0,0,2,0,1),\\
(6,0,0,3,1,0,0),(6,0,1,1,2,0,0),(7,0,0,0,3,0,0) \bigg\}. 
 \end{multline*}

\paragraph{Three conics and one line.}

We have $77$ possible weak-combinatorial types:
\begin{multline*}
C(n_{2},n_{3},t_{3},t_{5},t_{7},d_{6},d_{8}) \in \bigg\{
(0,0,0,0,0,2,2),(0,0,0,2,0,3,0),(0,0,1,0,1,3,0),\\(0,0,1,1,0,2,1),
(0,0,2,0,0,1,2),(0,0,2,2,0,2,0),(0,0,3,0,1,2,0),(0,0,3,1,0,1,1),\\(0,0,4,0,0,0,2),
(0,0,4,2,0,1,0),(0,0,5,0,1,1,0),(0,0,5,1,0,0,1),(0,0,6,2,0,0,0),\\(0,0,7,0,1,0,0),
(0,1,0,0,0,0,3),(0,1,0,1,1,2,0),(0,1,0,2,0,1,1),(0,1,1,0,1,1,1),\\(0,1,1,1,0,0,2),
(0,1,1,3,0,1,0),(0,1,2,1,1,1,0),(0,1,2,2,0,0,1),(0,1,3,0,1,0,1),\\(0,1,3,3,0,0,0),
(0,1,4,1,1,0,0),(0,2,0,0,2,1,0),(0,2,0,1,1,0,1),(0,2,0,4,0,0,0),\\(0,2,1,2,1,0,0),
(0,2,2,0,2,0,0),(1,0,0,0,1,2,1),(1,0,0,1,0,1,2),(1,0,0,3,0,2,0),\\(1,0,1,0,0,0,3),
(1,0,1,1,1,2,0),(1,0,1,2,0,1,1),(1,0,2,0,1,1,1),(1,0,2,1,0,0,2),\\(1,0,2,3,0,1,0),
(1,0,3,1,1,1,0),(1,0,3,2,0,0,1),(1,0,4,0,1,0,1),(1,0,4,3,0,0,0),\\(1,0,5,1,1,0,0),
(1,1,0,0,1,0,2),(1,1,0,2,1,1,0),(1,1,0,3,0,0,1),(1,1,1,0,2,1,0),\\(1,1,1,1,1,0,1),
(1,1,1,4,0,0,0),(1,1,2,2,1,0,0),(1,1,3,0,2,0,0),(1,2,0,1,2,0,0),\\(2,0,0,0,2,2,0),
(2,0,0,1,1,1,1),(2,0,0,2,0,0,2),(2,0,0,4,0,1,0),(2,0,1,0,1,0,2),\\(2,0,1,2,1,1,0),
(2,0,1,3,0,0,1),(2,0,2,0,2,1,0),(2,0,2,1,1,0,1),(2,0,2,4,0,0,0),\\(2,0,3,2,1,0,0),
(2,0,4,0,2,0,0),(2,1,0,0,2,0,1),(2,1,0,3,1,0,0),(2,1,1,1,2,0,0),\\(3,0,0,1,2,1,0),
(3,0,0,2,1,0,1),(3,0,0,5,0,0,0),(3,0,1,0,2,0,1),(3,0,1,3,1,0,0),\\(3,0,2,1,2,0,0),
(3,1,0,0,3,0,0),(4,0,0,2,2,0,0),(4,0,1,0,3,0,0) \bigg\}. 
\end{multline*}

\end{document}